\documentclass[preprint,12pt]{elsarticle}

\usepackage{amssymb}
\usepackage{amsmath}
\usepackage{amsthm}

\newtheorem{theor}{Theorem}[section]
\newtheorem{lemma}{Lemma}[section]

\newtheorem{remark}{Remark}[section]

\newcommand{\R}{\mathbb R}
\newcommand{\E}{\mathbb E}
\newcommand{\ucp}{\stackrel{ucp}{\rightarrow}}
\newcommand{\pn}{\stackrel{\mathbb P}{\rightarrow}}
\newcommand{\stab}{\stackrel{d_{st}}{\rightarrow}}
\newcommand{\schw}{\stackrel{d}{\rightarrow}}

\numberwithin{equation}{section}
\long\def\symbolfootnote[#1]#2{\begingroup\def\thefootnote{\fnsymbol{footnote}}
\footnote[#1]{#2}\endgroup}

\usepackage{a4wide}

\journal{Stochastic Processes and Their Applications}

\begin{document}

\begin{frontmatter}



\title{High-frequency asymptotics for path-dependent functionals of It\^o semimartingales}

\author{Moritz Duembgen\fnref{label1}}
\ead{m.duembgen@statslab.cam.ac.uk}
\address{University of Cambridge}
\author{Mark Podolskij\fnref{label2}}
\ead{m.podolskij@uni-heidelberg.de}
\address{Heidelberg University and CREATES}

\fntext[label1]{The financial support by Man Group plc and the EPSRC is greatly appreciated.}
\fntext[label2]{The financial support from CREATES is greatly appreciated.}

\begin{abstract}
The estimation of local characteristics of  It\^o semimartingales  has received a great deal of attention in both academia and industry over the past decades. In various papers limit theorems were derived for functionals of increments and ranges in the infill
asymptotics setting. 
In this paper we establish the asymptotic theory for a wide class of statistics that are built from the incremental process
of an It\^o semimartingale. More specifically, we will show the law of large numbers and the associated stable central limit theorem
for the path dependent functionals in the continuous and discontinuous framework. Some examples from economics and physics demonstrate
the potential applicability of our theoretical results in practice.
\end{abstract}


\end{frontmatter}

\vspace{4pt}
\section{Introduction}

In the last decade limit theory for high frequency observations of It\^o semimartingales has received 
a lot of attention in the scientific literature. Such observation scheme of semimartingales, also called \textit{infill asymptotics}, naturally
appears in financial, biological and physical applications among many others. For instance, a seminal work of Delbaen and Schachermayer
\cite{ds94} states that price processes must follow a semimartingale model under nor arbitrage conditions. \\
A general It\^o semimartingale exhibits a representation of the form
\begin{align*}
X_t = X_0 + \int_0^t\!\mu_s\,\text{d}s + \int_0^t\!\sigma_{s}\,\text{d}W_s +  \delta \boldsymbol 1_{\{|\delta | \le 1\}} \star (m - n)_t +   \delta \boldsymbol 1_{\{|\delta | > 1\}} \star m_t, 
\end{align*} 
where $\mu$ represents the drift, $\sigma$ is the volatility, $W$ is a Brownian motion, $m$ denotes the jump measure
associated with $X$ and $n$ is its compensator. Furthermore, for any measure $\pi$, we use the short hand notation $f \star \pi:=
\int f d\pi$ whenever the latter is well defined. Irrespective of the application field, researchers are interested in understanding the
fine structure of the underlying It\^o semimartingale model based on high frequency observations
\[
X_{0},X_{\Delta_n}, X_{2\Delta_n}, \ldots, X_{\Delta_n\lfloor t/\Delta_n\rfloor},
\]    
where $\Delta_n\rightarrow 0$, which refers to infill asymptotics. For various testing and estimation problems, the class of generalised
multipower variations turned out to be a very important probabilistic tool. In their most general form, generalised
multipower variations are defined as
\[
\sum_{i=1}^{\lfloor t/\Delta_n\rfloor -d+1} f\Big(a_n (X_{i\Delta_n} - X_{(i-1)\Delta_n}), \ldots, 
a_n (X_{(i+d-1)\Delta_n} - X_{(i+d-2)\Delta_n}) \Big),
\]
where $f:\R^d \rightarrow \R$ is a smooth function and the scaling $a_n$ depends on whether the process $X$ has jumps or not.
In the continuous case the proper scaling is $a_n=\Delta_n^{-1/2}$. Probabilistic properties of generalised
multipower variations in continuous and discontinuous settings
have been studies in \cite{bgjps06,jacod08,kp08} among many others. We refer to a recent book \cite{jp12} for a comprehensive 
study of high frequency asymptotics for It\^o semimartingales. Such probabilistic results found manifold applications 
in the statistical analysis of semimartingale models. Estimation of the quadratic variation (see e.g. \cite{jacod08}),  
volatility forecasting (see e.g. \cite{abm04,abm06}), and tests for the presence of the jump component 
(see e.g. \cite{aj09,bs06}) are the most prominent applications among many others. \\
The aim of this paper is to study the asymptotic behaviour 
of \textit{path dependent} high frequency functionals of It\^o semimartingales. This framework is motivated by the fact
that in some situations we can not directly observe the semimartingale $X$, but only its path dependent functional over
short time windows. Let us give two examples. In various applied sciences integrated diffusions (i.e. integrated It\^o semimartingales)
appear as a natural class of models for  a given random phenomena.
For example in physics, when a medium's surface (such as the arctic's sea ice) is modelled as a stochastic process, a sonar's measurement of the reflection of this surface is given by the local time of the surface's slope process (see e.g. \cite{n88, f12}). Since this local time process is typically an It\^o process again (see e.g.  \cite{r63, k63}), limit theorems for local averages are required in order to make inference on the structure of the original surface process (see e.g. \cite{f13} for a detailed discussion).
Because only discrete (high frequency) observations of such integrated diffusions are available, one can not recover the original 
path of the underlying It\^o semimartingales from it. Another example of path dependent functionals are ranges whose statistical
properties have been studied in \cite{gk80, parkinson80} in the case of law frequency observations of a scaled Brownian motion.
We also refer to an early result by William Feller \cite{feller51}, which characterises the distribution of the range of 
the Brownian motion. \\ 
In this paper we will consider functionals of the incremental process built from $X$, i.e.  
\begin{align*}
V(X,g)_t^n=\Delta_n \sum_{i=1}^{\lfloor t/\Delta_n\rfloor} g\Big(\big\{a_n \big(X_{(i-1\text{+}s)\Delta_n}-X_{(i-1)\Delta_n}\big);\, s\in[0,1]\big\}\Big),
\end{align*}
where $g$ is now operating on $C([0,1])$ and the scaling $a_n$ equals $\Delta_n^{-1/2}$ when $X$ is a continuous It\^o semimartingale.
Obviously, this class of statistics extends the classical concept of power variations to path dependent functionals. The function
$g(x)=\sup_{t\in [0,1]} x(t) - \inf_{t\in [0,1]} x(t)$ recovers the case of realised ranges as the have been considered in 
\cite{cp07,kp12} in the context of quadratic variation estimation. In this work we will prove the law of large numbers for the functional $V(X,g)_t^n$ and show the associated stable central limit theorem in continuous and discontinuous framework.
We remark that extending the analysis to general path dependent functionals increases the complexity of the proofs, 
which is due to the topological structure of the  space $C([0,1])$. 
Furthermore, a general asymptotic statement in the discontinuous case seems to be out of reach (in contrast to very general
results for classical power variations studied in \cite{jacod08}). For this reason, we restrict our attention to range statistics
of discontinuous It\^o semimartingales, as they seem to be useful in financial applications (see \cite{kp12}).
Finally, we present some applications of the probabilistic
results, in particular in the context of integrated diffusions and realised ranges.

The paper is organised as follows. In Section 2 we state the two main theorems for general functionals of continuous It\^o semimartingales, establishing the limits in probability as well as the associated stable central limit theorem. 
In Section 3 we apply the limit theory to three most prominent practical examples including general range statistics
and integrated diffusions. Section 4 is devoted to the limit theorems for realised ranges of discontinuous It\^o semimartingales.
The proofs of the main results are collected in Section 5.\vspace{8pt}

\section{Limit Theorems for Continuous It\^o Semimartingales}\vspace{4pt}
Before we present the main results we start by introducing some notation. We denote by $C([0,1])$ the space of continuous real valued
functions on the interval $[0,1]$, and by $\|\cdot\|_{\infty}$ the supremum  norm on $C([0,1])$. A function $f:C([0,1])\rightarrow \R$
is said to have polynomial growth if $|f(x)|\leq C(1+\|x\|_{\infty}^p)$ for some $C,p>0$. For any $x,y\in C([0,1])$
and $f:C([0,1])\rightarrow \R$, the expression $f'_{y}(x)$ denotes the G\^ateaux derivative of $f$ at point $x$ in the direction
of $y$, i.e. $f'_y(x):= \lim_{h\rightarrow 0}(f(x+hy) -f(x))/h$. \newline \newline
For any processes $Y^n, Y$  we denote by $Y^n \ucp Y$ the uniform convergence in probability, i.e. 
$\sup_{t\in [0,T]}|Y^n_t - Y_t| \pn 0$ for all $T>0$. Throughout this paper we frequently use the notion of stable convergence, which is due to Renyi
\cite{renyi63}. A sequence of random variables  $(Y_n)_{n\geq 1}$ on  
 $(\Omega, \mathcal F, \mathbb P)$ with values in a Polish space $(E, \mathcal{E})$
is said to converge stably in law to $Y$ ($Y_n \stab Y$), where $Y$ is defined on an extension
$(\Omega', \mathcal F', \mathbb P')$ of the original probability space, 
if and only if for any bounded, continuous function $f$ and any bounded $\mathcal{F}$-measurable random variable $Z$ it holds that
\begin{align} \label{defstable}
\E[ f(Y_n) Z] \rightarrow \E'[ f(Y) Z], \quad n \rightarrow \infty.
\end{align} 
Typically, we will deal with spaces $E=\mathbb D([0,T] , \R)$ equipped with the  uniform topology when the process $Y$ is continuous. Notice that stable convergence is a stronger mode of convergence than weak convergence. In fact, the statement 
$Y_n \stab Y$ is equivalent to the joint weak convergence $(Y_n, Z) \schw (Y,Z)$ for any $\mathcal F$-measurable random variable
$Z$.   

\subsection{Law of Large Numbers}

Throughout this section we are considering a stochastic process $X$ defined on a filtered probability space $(\Omega, \mathcal F, \mathbb F\!=\!(\mathcal F_t)_{t\ge 0}, \mathbb P)$ satisfying the usual conditions that follows the distribution of a diffusion
\begin{align*}
	X_t=X_0 + \int_0^t \mu_s\, \text{d}s + \int_0^t \sigma_{s}\, \text{d}W_s
\end{align*}
for $t\ge 0$, where $X_0$ is a constant, $W$ is a Brownian motion, $\mu$ is a predictable, locally bounded process and $\sigma$ is an adapted, c\'adl\'ag process. Given a function $g:C([0, 1]) \rightarrow \mathbb R$ and a vanishing sequence $(\Delta_n)_{n\in \mathbb N}$ we define the sequence of processes
\begin{align} \label{statistic}
	V(X, g)^n_t &:= \Delta_n \sum_{i=1}^{\left\lfloor t/\Delta_n \right\rfloor} g\big(\Delta_n^{-\frac 1 2}\ d_i^n (X)\big), \\[1.5 ex]
	\label{din}
	d_i^n (X)&:= \big\{X_{(i-1+s)\Delta_n} - X_{(i-1)\Delta_n}\big\}_{s\in [0,1]}.
\end{align}
For any $z\in \R$ and $g\in C([0,1])$ we introduce the quantity
\begin{align} \label{rho}
\rho_z(g):=\mathbb E\left[g(\left\{z\, W_s;\ s\in[0, 1]\right\})\right],
\end{align}
whenever the latter expectation is finite.
Our first result is the law of large numbers for the functional $V(X, g)^n_t$. \vspace{4pt}
\begin{theor}[Law of Large Numbers]\label{HFTheor1}
Let $g$ be a locally uniformly continuous functional, i.e. for $x,y \in C([0,1])$,
	\begin{enumerate}   \setlength{\itemsep}{8pt} 
		\item[(i)] given $K,\epsilon>0$ there exists $ \delta>0$ such that for $\lVert x\rVert_\infty, \lVert y\rVert_\infty \le K$, $\lVert x-y\rVert_\infty\le \delta$ it follows that $|g(x)-g(y)| \le \epsilon$,
		\end{enumerate} 
and have polynomial growth. Then it holds that
	\begin{align}\label{LLN}
		V(X, g)^n_t \  \stackrel{ucp}{\rightarrow} \ V(X,g)_t := \int_0^t \rho_{\sigma_s}(g) \,\text{d}s,
	\end{align}
	where the quantity $\rho_z(g)$ is defined at \eqref{rho}.
\end{theor}\vspace{6pt}
\begin{remark}\emph{
Our notion of locally uniform continuity is slightly unusual. Instead of requiring uniform continuity on neighbourhoods or compact sets we demand it on balls $B_{\le K}(0)=\{x\in C([0,1]);\ \lVert x\rVert_\infty \le K\}$ for $K>0$, which are not compact with respect to the uniform topology. This type of locally uniform continuity is not required in the classical limit theory for functionals
of increments of $X$ (see e.g. \cite{bgjps06}) since on finite dimensional spaces continuity on closed balls implies uniform
continuity. We remark that our locally uniform continuity assumption is satisfied whenever
\[
|g(x)-g(y)|\leq C\|x-y\|_{\infty}^\delta 
\]
for all $x,y\in C[0,1]$ and some $C,\delta >0$. This condition is satisfied for all practical examples.
}
\end{remark}\vspace{6pt}

\subsection{Central Limit Theorem}
Having determined the limit in probability we now turn to the associated stable central limit theorem. \vspace{4pt}
\begin{theor}[Central Limit Theorem]\label{HFTheor2}
Let $g$ satisfy the conditions of Theorem \ref{HFTheor1}. Moreover, we assume that
	\begin{enumerate} \setlength{\itemsep}{8pt}
		\item[(ii)] 	given $K, \epsilon>0$ there exists $\delta>0$ such that for $\lVert x\rVert_\infty, \lVert y\rVert_\infty \le K$, $\lVert x-y\rVert_\infty \le \delta, \lVert v\rVert_\infty \le 1$ it follows that $|g'_v(x)-g'_v(y)| \le \epsilon$,
		\item[(iii)]	there exist $C, p>0$ such that $|g'_v(x)| \le C(1+ \lVert x\rVert_\infty^p)$ for $\lVert v\rVert_\infty \le 1$.
	\end{enumerate}
	Let $\sigma$ be a continuous It\^o semimartingale of the form
	\begin{align*}
		\sigma_t = \sigma_0 + \int_0^t \tilde{\mu}_s \,\emph{d}s + \int_0^t \tilde \sigma_{s} \,\emph{d}W_s + \int_0^t \tilde v_{s}\, \emph{d}V_s, 
	\end{align*}
where $\tilde \mu,\ \tilde \sigma$ and $\tilde v$ are adapted, c\'adl\'ag processes and $V$ is another Brownian motion
independent of $W$.		
	Then it follows that
	\begin{align}\label{CLT}
		\Delta_n^{-\frac 1 2} \left(V(X, g)^n - V(X, g) \right) \ \stab \ U(X, g)
	\end{align}
	where  $U(X, g)_t := \int_0^t u^1_s \,\emph{d}s + \int_0^t u^2_s\,\emph{d}W_s + \int_0^t u^3_s \,\emph{d}W'_s$ with
	\begin{align*}
		u^1_s &:= \mu_s \rho_{\sigma_s}^{(2)}(g') + \frac 1 2 \tilde \sigma_s \rho_{\sigma_s}^{(3)}(g') - \frac 1 2 \tilde \sigma_s \rho_{\sigma_s}^{(2)}(g')\\
		u^2_s &:= \rho_{\sigma_s}^{(1)}(g), \\
		u^3_s &:= \sqrt{\rho_{\sigma_s}(g^2)-\rho^2_{\sigma_s}(g)-(\rho_{\sigma_s}^{(1)}(g))^2},
	\end{align*}
	and, for $z\in \R$ and $f(x,y):=g'_{y}(x)$,
	\begin{align*}
	 	\rho_z^{(1)}(g)&:=\mathbb E\Big[g\big(\left\{z\, W_s;\ s\in [0, 1]\right\}\big)\, W_1\Big], \\ 
	 	\rho_z^{(2)}(f) &:=\mathbb E\Big[f\big(\left\{z\, W_s;\ s\in [0, 1]\right\},\, \{s;\ s\in [0,1]\}\big)\Big], \\
	 	\rho_z^{(3)}(f) &:=\mathbb E\Big[f\big(\left\{z\, W_s;\ s\in [0, 1]\right\},\, \{W_s^2;\ s\in [0,1]\}\big)\Big].
	\end{align*}
Furthermore, $W'$ is a Brownian motion defined on an extension of $(\Omega, \mathcal F, \mathbb F, \mathbb P)$,
which is independent of $\mathcal F$.
\end{theor}
Some remarks on the application of this probabilistic result are in order.
\begin{remark}\emph{
When $g(x)\equiv f(x(1))$ for some function $f:\R\rightarrow \R$ such that $f,f'$ have polynomial growth, we recover 
the stable central limit theorem for functionals of increments of $X$. More precisely, it holds that
\[
\rho_z^{(1)}(g) = \E[f(zW_1) W_1], \quad \rho_z^{(2)}(g') = \E[f'(zW_1)], \qquad  \rho_z^{(3)}(g') = \E[f'(zW_1) W_1^2],
\]
and we obtain the one-dimensional analogue of the asymptotic theory presented in \cite{kp08}.
}
\end{remark}

\begin{remark} \label{rem2}\emph{
In general, Theorem \ref{HFTheor2} can not be applied for statistical inference, since the distribution of the limit
$U(X, g)_t$ is unknown. However, when $g$ is an even  functional, i.e. $g(x)=-g(x)$ for all $x\in C([0,1])$, things become different.
In this case it holds that 
\[
\rho_z^{(1)}(g) = \rho_z^{(2)}(g') =\rho_z^{(3)}(g') =0
\]  
for all $z\in \R$, since $W\stackrel{d}= -W$ and expectations of odd functionals of $W$ are $0$. Hence, the limiting process 
$U(X, g)$ has the form
\[
U(X, g)_t = \int_0^t \sqrt{\rho_{\sigma_s}(g^2)-\rho^2_{\sigma_s}(g)} \,\emph{d}W'_s, 
\] 
which is, conditionally on $\mathcal F$, a Gaussian martingale with mean $0$. For a fixed $t>0$, the result of 
Theorem \ref{HFTheor2} can be transformed into a standard central limit theorem when $g$ is even. A slight modification 
of Theorem \ref{HFTheor1} shows that 
\begin{align*}
V^n_t &:= \Delta_n \sum_{i=1}^{\left\lfloor t/\Delta_n \right\rfloor -1} \Big\{ g^2\big(\Delta_n^{-\frac 1 2}\ d_i^n (X)\big)
- g\big(\Delta_n^{-\frac 1 2}\ d_i^n (X)\big) g\big(\Delta_n^{-\frac 1 2}\ d_{i+1}^n (X)\big) \Big\} \\
&\ucp \int_0^t \rho_{\sigma_s}(g^2)-\rho^2_{\sigma_s}(g) \,\emph{d}s.
\end{align*}
(This should be compared with the asymptotic theory for bipower variation established in \cite{bgjps06}.) For any fixed $t>0$,
we then deduce a standard central limit theorem
\[
\frac{\Delta_n^{-\frac 1 2} \left(V(X, g)^n_t - V(X, g)_t \right)}{\sqrt{V^n_t}} \schw \mathcal N(0,1)
\]
by  properties of stable convergence. The latter can be used to obtain confidence regions for the quantity $V(X, g)_t$. 
}
\end{remark}

\section{Examples and Applications}\vspace{4pt}
In this section we present some examples that demonstrate the applicability of the limit theory for path dependent functionals
of continuous It\^o semimartingales. For comparison reasons we start with the classical results on power variations.

\paragraph{Example 1}
Here we consider the power variation case which corresponds to $g(x)\equiv f(x(1))$ with $f(x)=|x|^p$, $p>0$. Recalling the asymptotic 
theory from \cite{bgjps06} we conclude that  
\begin{align*}
\Delta_n^{1-\frac p 2}\sum_{i=1}^{\lfloor t/\Delta_n \rfloor} \big|X_{i\Delta_n}-X_{(i-1)\Delta_n}\big|^p
\ucp \lambda^{1, p} \int_0^t |\sigma_s|^p \,\text{d}s
\end{align*}
where $\lambda^{1, p}=\mathbb E[|W_1|^p]$. Moreover, the following stable central limit theorem holds 
\begin{align*}
		\Delta_n^{-\frac 1 2} \Bigg( \frac{\Delta_n^{1-\frac p 2}}{\lambda^{1, p}}  \sum_{i=0}^{\lfloor t/\Delta_n \rfloor} \big|X_{i\Delta_n}-X_{(i-1)\Delta_n}\big|^p -  \int_0^t\! |\sigma_s|^p\, \text{d}s\Bigg)
 \stab \sqrt{\Lambda^{1, p}} \int_0^t |\sigma_s|^p\, \text{d}W'_s,
\end{align*}
where $\Lambda^{1, p}:= \frac{\lambda^{1, 2p}-(\lambda^{1, p})^2}{(\lambda^{1, p})^2}$. 
Later on we will compare the efficiency of power variation with other estimators presented in the following examples.
\vspace{6pt}

\paragraph{Example 2}

Let $g:C([0,1])\rightarrow \mathbb R$ be defined as $g(x):=f(\int_0^1 x(s)\,\text{d}s)$ for a continuously
differentiable function  $f:\mathbb R \rightarrow \mathbb R$ such that $f,f'$ have polynomial growth. 
Then condition (i) of Theorem \ref{HFTheor1} is obviously satisfied. Furthermore, it holds that
\[
g'_y(x)= f'\left(\int_0^1\!x(s)\,\text{d}s \right) \int_0^1\! y(s)\,\text{d}s, \qquad \forall x,y\in C([0,1]),
\]
and conditions (ii) and (iii) of Theorem \ref{HFTheor2} are fulfilled since $f'$ is continuous and has polynomial growth. 
In particular, for $f(x)=|x|^p$ with $p>0$ we obtain that 
\begin{align*}
	\Delta_n^{1-\frac p 2}\sum_{i=1}^{\lfloor t/\Delta_n \rfloor} \Big|\Delta_n^{-1}\int_{(i-1)\Delta_n}^{i\Delta_n}\! X_s \,\text{d}s-X_{(i-1)\Delta_n}\Big|^p\quad \ucp \quad \lambda^{2, p} \int_0^t |\sigma_s|^p \,\text{d}s
\end{align*}
where $\lambda^{2, p}=\mathbb E[|\int_0^1 W_s\, \text{d}s|^p]$. Furthermore, for $p>1$  we deduce the corresponding stable central limit
theorem (cf. Remark \ref{rem2})
\begin{align*}
		\Delta_n^{-\frac 1 2} \Bigg( \frac{\Delta_n^{1-\frac p 2}}{\lambda^{2, p}} &\sum_{i=0}^{\lfloor t/\Delta_n \rfloor} \Big|\Delta_n^{-1}\int_{(i-1)\Delta_n}^{i \Delta_n}\!X_s\, \text{d}s -X_{(i-1)\Delta_n}\Big|^p - \int_0^t\! |\sigma_s|^p\, \text{d}s\Bigg) \\
		 &\stab \quad \sqrt{\Lambda^{2, p}} \int_0^t |\sigma_s|^p\, \text{d}W'_s,
\end{align*}
with $\Lambda^{2, p}:= \frac{\lambda^{2, 2p}-(\lambda^{2, p})^2}{(\lambda^{2, p})^2}$. \vspace{6pt}

\paragraph{Example 3}

Let us now consider the range-based functionals which has been originally studied in \cite{cp07}. 
Here the functional $g:C([0,1])\rightarrow \mathbb R$ is a function of the range, i.e. 
$g(x)=f(\sup_{t\in [0,\,1]}x(t) - \inf_{t\in [0,\,1]} x(t))$ for a continuously differentiable function 
$f:\mathbb R \rightarrow \mathbb R$, such that $f,f'$ have  polynomial growth. Then the law of large numbers
in Theorem \ref{HFTheor1} readily applies, but the central limit theorem cannot be directly deduced from Theorem
\ref{HFTheor2}, because the range is not G\^ateaux differentiable in general.  

However, we may apply the following result: Let $x,y\in C([0,1])$ be functions such that the set $M:=\{t\in [0,1]:~ t=\text{argmax}_{s\in [0,1]} x(s)\}$ is finite, then it holds that (cf. \cite{cp07})
\begin{align*}
	\frac 1 h \Big(\sup_{0\le s \le 1} \big(x(s)+ h y(s)\big) - \sup_{0\le s \le 1} x(s)\Big)= \max_{t\in M} y(t).
\end{align*}
In the proofs (see again \cite{cp07}) the function $x$ plays the role of the Brownian motion, which attains its maximum
(resp. minimum) at a unique point almost surely. Let $t_{max}:= \arg\max_{s\in [0,1]} W_s$ and $t_{min}:= \arg\min_{s\in [0,1]} W_s$.
Then the assertion of Theorem \ref{HFTheor2} remains valid 
 in the range case when $\sigma$ is everywhere invertible (cf. \cite{kp12}) with
\begin{align*}
 	\rho_x^{(1)}(g)	&=\mathbb E\Big[f\Big(x\, \Big(\sup_{0\le t \le 1}W_s - \inf_{0\le s \le 1} W_s\Big)\Big)\, W_1\Big], \\ 
 	\rho_x^{(2)}(g')	&=\mathbb E\Big[f'\Big(x\, \Big(\sup_{0\le t \le 1}W_s - \inf_{0\le s \le 1} W_s\Big)\Big)
\big(t_{max}-t_{min}\big)\Big], \\
 	\rho_x^{(3)}(g')	&=\mathbb E\Big[f'\Big(x\, \Big(\sup_{0\le t \le 1}W_s - \inf_{0\le s \le 1} W_s\Big)\Big)\big(W_{t_{max}}^2-W_{t_{min}}^2 \big)\Big],
\end{align*}
which extends the asymptotic theory presented in \cite{kp12} to general functions of the range. 
In particular, for $f(x)=|x|^p$ with $p>0$ we obtain that
\begin{align}
	\Delta_n^{1-\frac p 2}\sum_{i=1}^{\lfloor t/\Delta_n \rfloor} \sup_{s,u\in [(i-1)\Delta_n,\, i \Delta_n]} (X_s -X_u)^p  \quad \ucp\quad \lambda^{3, p}\ \int_0^t |\sigma_s|^p \,\text{d}s
\end{align}
where $\lambda^{3, p}=\mathbb E[\sup_{s,u\in [0,1]} (W_s -W_u)^p]$. Furthermore, since the function $f$
is even, we deduce the following central limit theorem 
\begin{align*}
		\Delta_n^{-\frac 1 2} \Bigg( \frac{\Delta_n^{1-\frac p 2}}{\lambda^{3, p}} &\sum_{i=0}^{\lfloor t/\Delta_n \rfloor} 
\sup_{s,u\in [(i-1)\Delta_n,\, i \Delta_n]} (X_s -X_u)^p - \int_0^t\! |\sigma_s|^p\, \text{d}s\Bigg) \\
		&\stab \quad \sqrt{\Lambda^{3, p}} \int_0^t |\sigma_s|^p\, dW'_s 
\end{align*}
where $\Lambda^{3, p}:= \frac{\lambda^{3, 2p}-(\lambda^{3, p})^2}{(\lambda^{3, p})^2}$. 
This recovers  the analysis presented in \cite{kp12}.\vspace{6pt}

\paragraph{Comparison of Examples 1-3}
When comparing different estimators of integrated powers of volatility presented in the previous examples,
we see that $\Lambda^{i, p},\ i=1,2,3$ serve as a convenient measure of their efficiency. We remark however
that this comparison is not fair as the sampling schemes of Example 1 and Examples 2-3 are not comparable. \newline

Since $\int_0^1 W_s \, \text d s \sim \mathcal N(0, 1/3)$, it follows that $$\lambda^{1, p} = 3^{p/2}\, \lambda^{2,p}$$ and $\Lambda^{1, p}$ coincides with $\Lambda^{2, p}$. However, $\Lambda^{3, p}$ is considerably smaller so as expected, range based estimation is asymptotically superior. For example in the case $p=2$ we find $\Lambda^{1, p},\Lambda^{2, p}= 2$, whereas $\Lambda^{3, p}\approx 0.4$. The smaller $p$ is, the more pronounced this relative difference becomes. Figure 1 illustrates these relationships.
\begin{figure}[!h]
	\centering
		\includegraphics[width=1\textwidth]{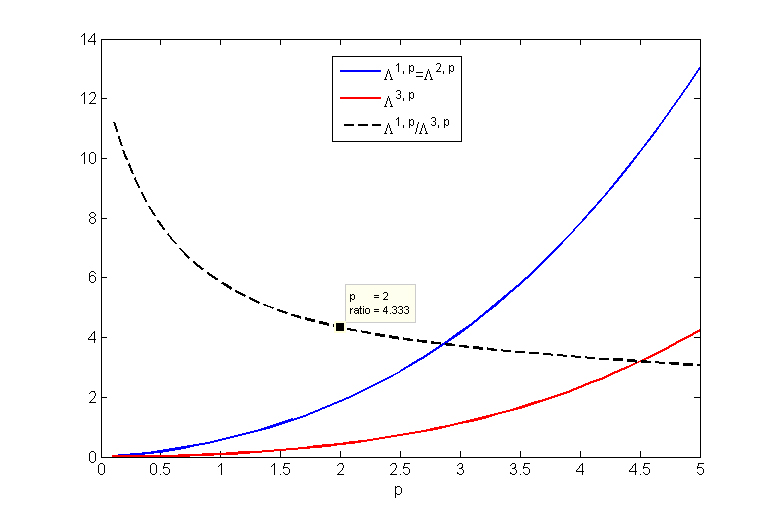}
		\caption{The parameters $\Lambda^{1, p}$=$\Lambda^{2,p}$, $\Lambda^{3,p}$ and their ratio.}
\end{figure}\vspace{6pt}

\paragraph{Example 4}
In various applied sciences integrated diffusions appear as a natural model of a random phenomena. For example in physics, when a medium's surface (such as the arctic's sea ice) is modelled as a stochastic process, a sonar's measurement of the reflection of this surface is given by the local time of the surface's slope process (see e.g. \cite{n88, f12}). Since this local time process is typically an It\^o process again (see e.g.  \cite{r63, k63}) and since the observations are given as local averages, limit theorems for local averages are required in order to make inference on the structure of the original surface process (see e.g. \cite{f13}).
So let's define the local averages of an It\^o process $X$ as 
\begin{align*}
	\overline{X}_i^n := \frac {1}{\Delta_n}\int_{(i-1)\Delta_n}^{i\Delta_n} X_s ds.
\end{align*}
A natural candidate estimator for the quadratic variation of $X$ is given by
\begin{align*}
	\sum_{i=2}^{\lfloor t/\Delta_n \rfloor} \left( \overline{X}_i^n - \overline{X}_{i-1}^n\right)^2.
\end{align*}
We note that this estimator does not directly exhibit a representation as in Example 2. However, when we use the decomposition
\begin{align*}
\overline{X}_i^n - \overline{X}_{i-1}^n &= \frac 1{\Delta_n} \Big( 
\int_{(i-1)\Delta_n}^{i\Delta_n} X_s - X_{(i-1)\Delta_n} ds - \int_{(i-2)\Delta_n}^{(i-1)\Delta_n} X_s - X_{(i-2)\Delta_n} ds  \\
&+(X_{(i-1)\Delta_n} - X_{(i-2)\Delta_n})\Big), 
\end{align*}
Theorem  \ref{HFTheor1}, and the bipower concept of Remark \ref{rem2}, we deduce the ucp convergence
\begin{align*}
\sum_{i=2}^{\lfloor t/\Delta_n \rfloor} \left( \overline{X}_i^n - \overline{X}_{i-1}^n\right)^2 
\ucp \frac 2 3 \int_0^t \sigma_s^2 ds.
\end{align*}
This clearly provides a way of estimating the quadratic variation of $X$ from observations of an integrated diffusion.

\section{Limit Theorems for It\^o Semimartingales with Jumps}

In this section we study the behavior of certain path-dependent functionals of discontinuous It\^o semimartingales. 
As the general theory is much more difficult to establish compared to the work of \cite{jacod08}, we restrict our attention
to ranges of It\^o semimartingales with jumps. For simplicity of exposition, we will further restrict ourselves to 
finite activity jump processes.

\subsection{Law of Large Numbers}

Consider now a stochastic process $X$ defined on a filtered probability space $(\Omega, \mathcal F, \mathbb F\!=\!(\mathcal F_t)_{t\ge 0}, \mathbb P)$ satisfying the usual conditions that follows the distribution of a diffusion with a jump component in the form of a compound Poisson process $Z_t=\sum_{i=1}^{N_t} J_i$ where $N$ is a Poisson process with intensity $\lambda$ and
i.i.d. jump sizes $J_i $, i.e.
\begin{align*}
	X_t=X_0 + \int_0^t \mu_s\, \text{d}s + \int_0^t \sigma_{s}\, \text{d}W_s + Z_t,
\end{align*}
where $W$ is a Brownian motion independent of $N$, $\mu$ is a predictable, locally bounded process and $\sigma$ is an adapted, c\'adl\'ag process. For a positive exponent $p>0$ we define
\begin{align*}
	R(X, p)^n_t := \sum_{i=1}^{\left\lfloor t/\Delta_n \right\rfloor}  \sup_{s,u\in [(i\!-\!1)\Delta_n, i\Delta_n]} |X_s - X_u|^p 
\end{align*}
for $t\ge 0,\ n\in \mathbb N$. 
Our first result is the following law of large numbers. 
\begin{theor}\label{HFTheor3}
	We have that
	\begin{align}\label{LLN2}
		R(X, p)^n_t \pn R(X, p)_t:=\begin{cases} 	
															 \lambda^{3,2} \int_0^t \sigma_s^2 \,\text{d}s + \sum_{i=1}^{N_t} J_i^2 & p = 2 \\ 
															 \sum_{i=1}^{N_t} |J_i|^p & p>2\end{cases}
	\end{align}
	where $\lambda^{3, 2}=\mathbb E[\sup_{s,u\in [0,\,1]} |W_s -W_u|^2]$.
\end{theor}
For $p<2$ we obtain infinity in the limit whenever $\int_0^t \sigma_s^2\, \text{d}s>0$. We remark that the first convergence
of Theorem \ref{HFTheor3} has been already proved in \cite{kp12} in the context 
of range based estimation of quadratic variation. Very similar results has been established for the 
classical power variations in \cite{jacod08}.

\subsection{Central Limit Theorem}
Having determined the limit in probability we now turn our attention to the associated stable central limit theorems.
In order to introduce the weak limit theory we require some further notation. We denote by $(T_i)_{i\geq 1}$ 
the successive jump times of the Poisson process $N$. Furthermore, we introduce two Brownian motions $(W'_t)_{t\geq 0}$,
$(\widetilde{W}_t)_{t\geq 0}$ and a sequence $(\kappa_i)_{i\geq 1}$ of i.i.d. $\mathcal U([0,1])$-distributed random
variables, which are mutually independent, and independent of $\mathcal F$. Finally, we introduce the process
\begin{align}
U(X, p)_t &=\! p \sum_{i=1}^{N_t} |J_i|^{p-1} \left \{ \sup_{\substack{0\le s \le \kappa_i \\ \kappa_i \le u \le 1}}\!
\left( (\widetilde{W}_{i\text{+}\kappa_i}\!-\!\widetilde{W}_{i\text{+}s})\sigma_{T_i-}\!+\! 
(\widetilde{W}_{i\text{+}u}\!-\!\widetilde{W}_{i\text{+}\kappa_i})\sigma_{T_i}\right) \boldsymbol 1_{\{J_i>0\}} \right. 
\nonumber \\[1.5 ex]
\label{Udef}
&\left. + \sup_{\substack{0\le s \le \kappa_i \\ \kappa_i \le u \le 1}}\!
\left( -(\widetilde{W}_{i\text{+}\kappa_i}\!-\!\widetilde{W}_{i\text{+}s})\sigma_{T_i-}\!-\! 
(\widetilde{W}_{i\text{+}u}\!-\!\widetilde{W}_{i\text{+}\kappa_i})\sigma_{T_i}\right) \boldsymbol 1_{\{J_i<0\}} \right \}
\end{align}
that is defined on the extension of the original space 
$(\Omega, \mathcal F, \mathbb F, \mathbb P)$. The central limit theorem is as follows.
\vspace{4pt}
\begin{theor}[Central Limit Theorem]\label{HFTheor4} \textit{} \newline
(i)	For $p>3$ and fixed $t>0$ we obtain the stable convergence
	\begin{align}\label{CLT2p3}
		\Delta_n^{-\frac 1 2} \left(R(X, p)^n_t - R(X, p)_t\right) \ \stab \ U(X,p)_t.
	\end{align}
(ii) Let $p=2$. Assume that the invertible volatility process $\sigma$ follows the distribution of a discontinuous It\^o semimartingale 
	\begin{align*}
		\sigma_t = \sigma_0 + \int_0^t \tilde{\mu}_s \,\emph{d}s + \int_0^t \tilde \sigma_{s} \,\emph{d}W_s + \int_0^t \tilde v_{s}\, \emph{d}V_s + \tilde Z_t, 
	\end{align*}
where $\tilde \mu,\ \tilde \sigma$ and $\tilde v$ are adapted, c\'adl\'ag processes, $V$ is another Brownian motion 
independent of $W$ and $\tilde Z_t= \sum_{i=0}^{\tilde N_t} \tilde J_i$ is a compound Poisson processes with $\tilde N$
being independent of $W$ ($\tilde N$ and $N$ are possibly
correlated). Then, for any fixed $t>0$, we obtain the stable convergence
	\begin{align}\label{CLT2p2}
		\Delta_n^{-\frac 1 2} \left(R(X, 2)^n_t - R(X, 2)_t\right) \stab U(X,2)_t + 
\sqrt{\lambda^{3, 4}-(\lambda^{3, 2})^2} \int_0^t \sigma_s^2\, dW'_s. 
	\end{align}
\end{theor}
We remark that Theorem \ref{HFTheor4} is similar in fashion to central limit theorems for classical power variations; see 
\cite{jacod08}. We do believe that \ref{HFTheor4} remains valid for a rather general It\^o semimartingale model (i.e. not only
in the finite activity case), but the proofs become considerably longer. 

After local estimation of $\sigma $ and  jump sizes $J_i$, the \textit{conditional} law of $U(X, p)_t$ given $\mathcal F$ can be simulated. However,
unlike for the mixed normal case in the classical power variation framework, the knowledge of the conditional law of $U(X, p)_t$
is not sufficient for statistical inference (e.g. construction of confidence regions).

\section{Proofs}\vspace{4pt}
First of all, note that without loss of generality we may assume that the processes $\mu,\sigma , \tilde{\mu}, \tilde{\sigma },
\tilde{v}$ are bounded. This follows from a standard localization procedure (see e.g. \cite{bgjps06}). Below, all positive constants
are denoted by $C$ or $C_p$ if they depend on an external parameter $p$, although they may change from line to line.

\paragraph{Proof of Theorem \ref{HFTheor1}}
We begin with some preliminary observations. Denoting 
\[
A_t:= \int_{0}^t\! \mu_s \,\text{d}s, \qquad M_t:= \int_{0}^t\! \sigma_s \,\text{d}W_s,
\]
we find that for $p>0$,
	\begin{align} \nonumber
		\mathbb E\left[\left(\Delta_n^{-\frac 1 2}\ \lVert d_i^n (X)\rVert_\infty\right)^p\right] &\le 
		C_p\Delta_n^{-\frac p 2}\left(\mathbb E\Big[\lVert d_i^n (A)\rVert_\infty^p\Big] + \mathbb E\Big[\lVert d_i^n (M)\rVert_\infty^p\Big]\right) \\ &\le
		C_p\left(\Delta_n^{\frac p 2} \lVert\mu\rVert_\infty^p  + \Delta_n^{-\frac p 2}\mathbb E\bigg[\Big(\int_{(i-1)\Delta_n}^{i\Delta_n} \sigma_s^2 \,\text{d}s\Big)^{\frac p 2}\bigg]\right) \nonumber \\ \label{Deltainfinmom} &\le
		C_p\left(\Delta_n^{\frac p 2}\lVert\mu\rVert_\infty^p + \lVert\sigma\rVert_\infty^p\right)  < \infty
	\end{align}
	where we used the Burkholder-Davis-Gundy inequality and the boundedness of $\mu$ and $\sigma$. Now, by the assumption of polynomial growth, $\left|g(x)\right| \le C (1+ \lVert x\rVert_{\infty}^p)$ for $p>0$ so
	\begin{align} \label{gDeltainfinmom}
		\mathbb E\left[ g(\Delta_n^{-\frac 1 2} \ d_i^n (X))\right] &\le 
		C(1+ \Delta_n^{-\frac p 2}\mathbb E\Big[ \lVert d_i^n (X)\rVert_\infty^p\Big]) < \infty.
	\end{align}
	Define $\beta_i^n:= \Delta_n^{-\frac 1 2} \sigma_{(i-1)\Delta_n}\ d_i^n (W)$, an approximation 
of $\Delta_n^{-\frac 1 2}\ d_i^n (X)$. As in \eqref{Deltainfinmom}, \eqref{gDeltainfinmom} we find that 
	\begin{align}
		\mathbb E\Big[ \lVert\beta_i^n\rVert_\infty^p\Big] &\le C_p ,\quad p>0, \label{betainfinmom} \\
		\mathbb E\Big[\big|g(\beta_i^n)\big|\Big] &\le C. \label{gbetainfinmom}
	\end{align}
	$\beta_i^n$ will serve as a convenient approximation because of its simple form and
	\begin{align}\nonumber
		&\mathbb E\bigg[ \lVert\beta_i^n  -\Delta_n^{-\frac 1 2} d_i^n(X)\rVert_\infty^p\bigg]  \\ \nonumber&=
		\Delta_n^{-\frac p 2}\, \mathbb E\bigg[\sup_{[(i-1)\Delta_n, i\Delta_n]} \Big| \int_{(i-1)\Delta_n}^t\! \mu_s \,\text{d}s + \int_{(i-1)\Delta_n}^t\! \left(\sigma_s-\sigma_{(i-1)\Delta_n}\right) \,\text{d}W_s\Big|^p\bigg] \\ \nonumber&\le
		C \left(\lVert\mu\rVert_\infty^p \Delta_n^{\frac p 2}+ \Delta_n^{- \frac p 2}\, \mathbb E\bigg[\Big(\int_{(i-1)\Delta_n}^{i\Delta_n}\! \left(\sigma_s-\sigma_{(i-1)\Delta_n}\right)^2 \,\text{d}s\Big)^{\frac p 2}\bigg]\right) \\\label{betainapprox1} &\rightarrow
		0,
	\end{align}
	where we used again the Burkholder-Davis-Gundy inequality and for the last step that $\sigma$ is c\'adl\'ag. Returning to the claimed convergence in ucp, let
	\begin{align*}
		U^n_t &:= \Delta_n \sum_{i=1}^{\left\lfloor t/\Delta_n \right\rfloor}  \mathbb E\left[g(\beta_i^n) \left| \ \mathcal F_{(i-1)\Delta_n}\right.\right],  \\ 
		R^{1,n}_t &:= \Delta_n \sum_{i=1}^{\left\lfloor t/\Delta_n \right\rfloor}  \big(g(\beta_i^n)- \mathbb E\left[g(\beta_i^n) \left| \ \mathcal F_{(i-1)\Delta_n}\right.\right]\big), \\
		R^{2,n}_t &:= \Delta_n \sum_{i=1}^{\left\lfloor t/\Delta_n \right\rfloor}  \big(g(\Delta_n^{-\frac 1 2}\ d_i^n (X))-g(\beta_i^n)\big),
	\end{align*}
	for all $t\ge 0,\ n\in\mathbb N$. Clearly, $V(X, g)^n_t= U^n_t + R^{1,n}_t+R^{2,n}_t$. In order to prove \eqref{LLN}, we will first show that the approximation $U^n$ converges to $V(X, g)$ and afterwards that the error terms $R^1,\ R^2$ vanish. By definition, $\mathbb E\left[g(\beta_i^n) \left|\ \mathcal F_{(i-1)\Delta_n}\right.\right]=\rho_{\sigma_{(i-1)\Delta_n}}\left(g\right)$, and  therefore
	\begin{align*}
		 U^n_t = \Delta_n \sum_{i=1}^{\left\lfloor t/\Delta_n \right\rfloor} \rho_{\sigma_{(i-1)\Delta_n}}\left(g\right)
		 \ucp V(X, g)_t  = \int_0^{t} \rho_{\sigma_s}(g) \,\text{d}s
	\end{align*} 
	due to continuity of the function $\rho (g)$.  Turning to the claimed disappearance of $R^{1, n}$ we exploit its martingale property and apply Doob's maximal inequality to get that
	\begin{align*}
		\mathbb P\left[\sup_{0\le t \le T} \left| R^{1, n}_t\right| > \epsilon\right] \leq 
		C \frac{\Delta_n^2}{\epsilon^2} \sum_{i=1}^{\left\lfloor T/\Delta_n\right\rfloor} \mathbb E\left[g(\beta_i^n)^2\right] 
       \leq     C_T \Delta_n \epsilon^{-2} \rightarrow 0 
	\end{align*}
	for each $\epsilon>0$. Regarding $R^{2, n}$, Chebyshev's inequality gives that for $\epsilon >0$,
	\begin{align*}
		\mathbb P\bigg[\sup_{0 \le t \le T}\left|R^{2, n}_t\right|> \epsilon\bigg] \le 
		\frac{\Delta_n}{\epsilon} \sum_{i=1}^{\left\lfloor T/\Delta_n \right\rfloor} \mathbb E\left[\big| g(\Delta_n^{-\frac 1 2}\ d_i^n (X))-g(\beta_i^n)\big|\right].
	\end{align*}
	Now we make use of the locally uniform continuity of $g$. For $K, \hat \epsilon>0$ choose $\delta>0$ as in \textit{(i)}. Defining $A^{i, n, K}\!:=\!\{\lVert \beta_i^n\rVert _\infty + \lVert \Delta_n^{-\frac 1 2}\ d_i^n (X)\rVert _\infty \le K\}$ as well as $A^{i, n, K, \delta}\!:=A^{i, n, K}\cap \{ \lVert \beta_i^n - \Delta_n^{-\frac 1 2}\ d_i^n (X)\rVert _\infty \le \delta\}$ and denoting $\Delta_i^n g:=g(\Delta_n^{-\frac 1 2}\ d_i^n (X))-g(\beta_i^n)$, we find that
	\begin{align}\label{localunifcontvanish}
		\mathbb E\left[|\Delta_i^n g|\right] &= 
		\mathbb E\left[|\Delta_i^n g |\left( \boldsymbol 1_{A^{i,n,K,\delta}} + \boldsymbol 1_{A^{i,n,K}\setminus A^{i,n,K,\delta}} + \boldsymbol 1_{\Omega \setminus A^{i,n, K}}\right)\right] \\ &\le 
		\hat \epsilon + C \Big( \mathbb E\left[  \lVert \beta_i^n - \Delta_n^{-\frac 1 2}\ d_i^n (X)\rVert _\infty \right] / \delta +  1 / K \Big)\nonumber
	\end{align}
	Hence, choosing $K$ and $n$ large, and then $\hat \epsilon$ small, we see that 
$\mathbb P\left[\sup_{0 \le t \le T}|R^{2, n}_t|> \epsilon\right]$ vanishes as $n\rightarrow \infty$ and we are done. \qed \vspace{4pt}

\paragraph{Proof of Theorem \ref{HFTheor2}}
Thanks to $\sigma$ following a diffusion process the approximation $\beta_i^n$ is now sharper than in \eqref{betainapprox1}:
\begin{align}\nonumber
	\Delta_n^{-\frac p 2}\, &\mathbb E\bigg[ \lVert\beta_i^n-\Delta_n^{-\frac 1 2} d_i^n(X)\rVert_\infty^p\bigg] \\ \nonumber &=
	\Delta_n^{- p}\, \mathbb E\bigg[\sup_{t\in [(i-1)\Delta_n, i\Delta_n]} \Big| \int_{(i-1)\Delta_n}^t\!\mu_s \,\text{d}s + \int_{(i-1)\Delta_n}^t \!\left(\sigma_s-\sigma_{(i-1)\Delta_n}\right) \,\text{d}W_s\Big|^p\bigg] \\  \nonumber &\le
	\Delta_n^{-p}\, C \bigg(\lVert\mu\rVert_\infty^p \Delta_n^p+ C\, \mathbb E\bigg[\Big(\int_{(i-1)\Delta_n}^{i\Delta_n}\! \left(\sigma_s-\sigma_{(i-1)\Delta_n}\right)^2 \,\text{d}s\Big)^{\frac p 2}\bigg]\bigg) \\ \label{betainapprox2} &\le
	C \Big(\lVert\mu\rVert_\infty^p + \Delta_n^{-\frac p 2} \mathbb E\Big[\sup_{s\in [(i-1)\Delta_n,i\Delta_n]}|\sigma_s-\sigma_{(i-1)\Delta_n}|^p\Big]\Big) \\ \nonumber&\le
	C \bigg(\lVert\mu\rVert_\infty^p +  \Delta_n^{\frac p 2} \lVert \tilde \mu\rVert _\infty + \Delta_n^{-\frac p 2} \mathbb E\bigg[  \Big(\int_{(i-1)\Delta_n}^{i\Delta_n} \!(\tilde \sigma_s-\tilde \sigma_{(i-1)\Delta_n})^2\,\text{d}s\Big)^{\frac p 2}\bigg]\bigg) \\ \nonumber & \le 
	C.
\end{align}
Again, we used the Burkholder-Davis-Gundy inequality and the c\'adl\'ag property of $\tilde \sigma$. 
In order to prove \eqref{CLT} we split up the original term $\Delta_n^{-\frac 1 2}\left(V(X,g)^n_t-V(X,g)_t\right)$ 
into an approximation and several error terms:
\begin{align*}
	\Delta_n^{-\frac 1 2}\left(V(X,g)^n_t-V(X,g)_t\right) &= 
	\Delta_n^{-\frac 1 2} \left(\Delta_n \sum_{i=1}^{\left\lfloor t/\Delta_n\right\rfloor} g(\Delta_n^{-\frac 1 2} d_i^n (X)) - \int_0^t \rho_{\sigma_s}(g)\,\text{d}s\right) \\ &=: 
	U^n_t + R^{1, n}_t + R^{2, n}_t + R^{3, n}_t + R^{4, n}_t,
\end{align*}
where
\begin{align*}
	U^n_t &= \Delta_n^{\frac 1 2} \sum_{i=1}^{\left\lfloor t/\Delta_n \right\rfloor} \left( g(\beta_i^n) - \mathbb E\left[g(\beta_i^n)\left|\ \mathcal F_{(i-1) \Delta_n}\right.\right]\right),  \\
	R^{1, n}_t &= \Delta_n^{\frac 1 2} \sum_{i=1}^{\left\lfloor t/\Delta_n \right\rfloor} \left(g(\Delta_n^{-\frac 1 2} d_i^n(X))-g(\beta_i^n) - \mathbb E\left[g(\Delta_n^{-\frac 1 2} d_i^n(X)) -g(\beta_i^n)\left|\ \mathcal F_{(i-1) \Delta_n}\right.\right]\right),  \\
	R^{2, n}_t &= \Delta_n^{\frac 1 2} \sum_{i=1}^{\left\lfloor t/\Delta_n \right\rfloor} \mathbb E\left[g(\Delta_n^{-\frac 1 2} d_i^n(X)) -g(\beta_i^n)\left|\ \mathcal F_{(i-1) \Delta_n}\right.\right],  \\
	R^{3, n}_t &= \Delta_n^{-\frac 1 2} \sum_{i=1}^{\left\lfloor t/\Delta_n\right\rfloor} \left(\Delta_n \mathbb E\left[g(\beta_i^n)\left|\ \mathcal F_{(i-1) \Delta_n}\right.\right] - \int_{(i-1)\Delta_n}^{i\Delta_n} \rho_{\sigma_s}(g)\,\text{d}s\right),\\
	R^{4, n}_t &= \Delta_n^{-\frac 1 2} \int_{\left\lfloor t/\Delta_n\right\rfloor\Delta_n}^t \rho_{\sigma_s}(g)\,\text{d}s.
\end{align*}

Obviously, $R^{4,n}_t \stackrel{ucp}{\rightarrow} 0$ due to the boundedness of $\sigma$ and the continuity of $\rho$. Furthermore, we also have
\begin{lemma}
Under conditions of Theorem \ref{HFTheor2} we obtain
	\begin{itemize}\setlength{\itemsep}{8pt}
		\item[(i)] 
			$U^n_t\ \stab\ \int_0^t \rho_{\sigma_s}^{(1)}(g)\,\emph{d}W_s + \int_0^t \sqrt{\rho_{\sigma_s}(g^2)-\rho_{\sigma_s}(g)^2-(\rho_{\sigma_s}^{(1)}(g))^2}\, \emph{d}W'_s $,
		\item[(ii)] 
			$R^{1,n}_t\ \stackrel{ucp}{\rightarrow}\ 0$, 
		\item[(iii)]
			$R^{2,n}_t\ \stackrel{ucp}{\rightarrow}\ \int_0^t \mu_s \rho_{\sigma_s}^{(2)}(g') \,\emph{d}s + \frac 1 2 \int_0^t \tilde \sigma_s \rho_{\sigma_s}^{(3)}(g') \,\emph{d}s  - \frac 1 2 \int_0^t \tilde \sigma_s \rho_{\sigma_s}^{(2)}(g')\,\emph{d}s$,
		\item[(iv)]
			$R^{3,n}_t\ \stackrel{ucp}{\rightarrow}\ 0$.
	\end{itemize}
\end{lemma}
\begin{proof}\mbox{}
	\begin{itemize}
		\item[\textit{(i)}] 
			Defining $\xi_i^n :=\Delta_n^{\frac 1 2}(g(\beta_i^n)-\mathbb E[g(\beta_i^n)\left|\ \mathcal F_{(i-1)\Delta_n}\right.])$ we have $U^n_t= \sum_{i=1}^{\left\lfloor t/\Delta_n \right\rfloor} \xi_i^n$. Now we will verify the conditions of Jacod's theorem of stable convergence for semimartingales (see \cite{jacod94}). Introducing the notation  $\Delta_i^n  W := W_{i\Delta_n}-W_{(i-1)\Delta_n}$ we find that
			\begin{align*}
				\mathbb E[\xi_i^n\left| \ \mathcal F_{(i-1) \Delta_n}\right.] &= 0,
			\end{align*}
\begin{align*}	
				\sum_{i=1}^{\left\lfloor t/\Delta_n \right\rfloor} \mathbb E[(\xi_i^n)^2\left|\ \mathcal F_{(i-1) \Delta_n}\right.] &= 
				\Delta_n \sum_{i=1}^{\left\lfloor t/\Delta_n \right\rfloor}\left( \mathbb E[g(\beta_i^n)^2\left| \ \mathcal F_{(i-1) \Delta_n}\right.] - \mathbb E[g(\beta_i^n)\left| \ \mathcal F_{(i-1) \Delta_n}\right.]^2\right) \\ &=
				\Delta_n \sum_{i=1}^{\left\lfloor t/\Delta_n \right\rfloor} \left(\rho_{\sigma_{(i-1)\Delta_n}}(g^2)-\rho_{\sigma_{(i-1)\Delta_n}}(g)^2\right) \\ &\ucp
				\int_0^t \left(\rho_{\sigma_s}(g^2)-\rho_{\sigma_s}(g)^2\right) \,\text{d}s , 
\end{align*}
\begin{align*}				
				\sum_{i=1}^{\left\lfloor t/\Delta_n \right\rfloor} \mathbb E[\xi_i^n\, \Delta_i^n W \left|\ \mathcal F_{(i-1) \Delta_n}\right.] &=
				\Delta_n^{\frac 1 2} \sum_{i=1}^{\left\lfloor t/\Delta_n \right\rfloor} \mathbb E[g(\beta_i^n)\, \Delta_i^n W \left|\ \mathcal F_{(i-1) \Delta_n}\right.] \\ &=
				\Delta_n \sum_{i=1}^{\left\lfloor t/\Delta_n \right\rfloor} \left.\mathbb E[g(\left\{x\,W_s;\ s\in [0, 1]\right\}) W_1]\right|_{x=\sigma_{(i-1)\Delta_n}} \\ &\ucp
				\int_0^t\rho_{\sigma_s}^{(1)}(g) \,\text{d}s ,
\end{align*}
\begin{align*}				
				\sum_{i=1}^{\left\lfloor t/\Delta_n \right\rfloor} \mathbb E[(\xi_i^n)^2\boldsymbol 1_{\left\{\left|\xi_i^n\right|> \epsilon\right\}} \left|\ \mathcal F_{(i-1) \Delta_n}\right.] &\le
				\frac{\Delta_n^2}{ \epsilon^2} \sum_{i=1}^{\left\lfloor t/\Delta_n \right\rfloor} \mathbb E[(g(\beta_i^n)-\mathbb E[g(\beta_i^n)\left|\ \mathcal F_{(i-1) \Delta_n}\right.])^4\left| \ \mathcal F_{(i-1) \Delta_n}\right.] \\ &\le
				\Delta_n C / \epsilon^2 \rightarrow 0.
			\end{align*}
			
			Finally, let $N\in \mathcal M_b(W)^\bot$, the space of all bounded $(\mathbb P, \mathbb F)$-martingales that have zero quadratic covariation with $W$. 
Define 
$M_u:= \mathbb E[g(\beta_i^n) | \mathcal F_{u}]$ for $u\geq (j-1)\Delta_n$. By the martingale representation theorem we deduce the
identity
\begin{align*} 
M_u = M_{(i-1)\Delta_n} + \int_{(i-1)\Delta_n}^u \eta_s \, \text{d}W_s
\end{align*}    
for a suitable predictable process $\eta$. By the It\^o isometry we conclude that  
\begin{align*} 
\mathbb E[g(\beta_i^n) \Delta_i^n N | \mathcal F_{(i-1) \Delta_n}] &= 
\mathbb E[M_{i\Delta_n} \Delta _i^n N | \mathcal F_{(i-1) \Delta_n}] \\
&= \mathbb E[ \Delta_i^n M \Delta_i^n N | \mathcal F_{(i-1) \Delta_n}] =0.
\end{align*}
Hence, Jacod's convergence theorem (see \cite[Theorem IX.7.28]{JS}) gives 
			\begin{align*}
				U^n_t\ \stab\ \int_0^t \rho_{\sigma_s}^{(1)}(g)\,\text{d}W_s + \int_0^t 
\sqrt{\rho_{\sigma_s}(g^2)-\rho_{\sigma_s}(g)^2-(\rho_{\sigma_s}^{(1)}(g))^2}\, \text{d}W'_s.
			\end{align*}
			\qed
			
		\item[\textit{(ii)}] 
			Let $\eta_i^n := \Delta_n^{\frac 1 2}(g(\Delta_n^{-\frac 1 2}\,d_i^n (X))-g(\beta_i^n))$ so $R^{1, n}_t = \sum_{i=1}^{\left\lfloor t/\Delta_n \right\rfloor} (\eta_i^n - \mathbb E[\eta_i^n |\ \mathcal F_{(i-1) \Delta_n}])$. Since $R^{1, n}$ is a martingale we may apply Doob's inequality to obtain
			\begin{align*}
				\mathbb P\left[\sup_{t \le T}\left|R^{1, n}_t\right| > \epsilon\right] \le 
				C\frac{\Delta_n}{\epsilon^2}   \sum_{i=1}^{\left\lfloor T/\Delta_n \right\rfloor} \mathbb E\Big[\big(g(\Delta_n^{-\frac 1 2}\,d_i^n(X))-g(\beta_i^n)\big)^2\Big].
			\end{align*}
			By the same argument as in \eqref{localunifcontvanish}, making use of the locally uniform continuity of $g$, we find that the last term converges to $0$. \qed
			
		\item[\textit{(iii)}] 
			By the assumed G\^ateaux  differentiability of $g$ the mean-value theorem gives 
			\begin{align*}
				g(y)-g(x) = 
				g'_{y-x}(x+\hat t(y-x)) 
			\end{align*}
			for some $\hat t\in [0,1]$. Let us again use the notation $f(x;y):=g'_y(x)$.  
We expand $R^{2, n}= R^{2.1, n}+R^{2.2, n}$ where
			\begin{align*}
				R^{2.1, n}_t&:= \Delta_n^{\frac 1 2} \sum_{i=1}^{\left\lfloor t/\Delta_n \right\rfloor} \mathbb E\left[f'(\beta_i^n;\ \Delta_n^{-\frac 1 2} d_i^n(X) -\beta_i^n) \left|\ \mathcal F_{(i-1)\Delta_n}\right.\right], \\ 
				R^{2.2, n}_t&:= \Delta_n^{\frac 1 2} \sum_{i=1}^{\left\lfloor t/\Delta_n \right\rfloor} \mathbb E\left[\big(f'(\chi_i^n;\ \Delta_n^{-\frac 1 2} d_i^n(X) -\beta_i^n)-f'(\beta_i^n;\ \Delta_n^{-\frac 1 2} d_i^n(X) -\beta_i^n)\big) \left|\ \mathcal F_{(i-1)\Delta_n}\right.\right],
			\end{align*}
			with $\chi_i^n= \beta_i^n + \hat t_i^n (\Delta_n^{-\frac 1 2}\, d_i^n(X) -\beta_i^n)$ and $\hat t_i^n \in [0, 1]$. Decompose also $\Delta_n^{-\frac 1 2}\, d_i^n(X) -\beta_i^n= V_i^n(1) + V_i^n(2)$ where
			\begin{align*}
				V_i^n(1)_t &:= \Delta_n^{-\frac 1 2}\bigg(t\, \Delta_n\, \mu_{(i-1)\Delta_n} \\
				&+\int_{(i-1)\Delta_n}^{(i-1\text{+}t)\Delta_n}\!\Big({\tilde \sigma}_{(i-1)\Delta_n}(W_s - W_{(i-1)\Delta_n}) + {\tilde v}_{(i-1)\Delta_n}(V_s - V_{(i-1)\Delta_n})\Big)\,\text{d}W_s\bigg) \\ &=
				\Delta_n^{-\frac 1 2}\bigg(t\, \Delta_n\, \mu_{(i-1)\Delta_n}+ \frac 1 2\, {\tilde \sigma}_{(i-1)\Delta_n}((W_{(i-1\text{+}t)\Delta_n}-W_{(i-1)\Delta_n})^2-t\, \Delta_n) \\ &
				+{\tilde v}_{(i-1)\Delta_n}\int_{(i-1)\Delta_n}^{(i-1\text{+}t)\Delta_n} (V_s - V_{(i-1)\Delta_n})\,\text{d}W_s\bigg),  \\				
				V_i^n(2)_t :&= 
				\Delta_n^{-\frac 1 2}\Bigg(\int_{(i-1)\Delta_n}^{(i-1\text{+}t)\Delta_n} \big(\mu_s-\mu_{(i-1)\Delta_n}\big)\,\text{d}s\  +  \Big(\int_{(i-1)\Delta_n}^s\big({\tilde v}_u-{\tilde v}_{(i-1)\Delta_n}\big)\text{d}V_u\Big) \,\text{d}W_s \\ &
				+\Big(\int_{(i-1)\Delta_n}^s{\tilde \mu}_u\, \text{du }+ \int_{(i-1)\Delta_n}^s({\tilde \sigma}_u-{\tilde \sigma}_{(i-1)\Delta_n})\,\text{d}W_u\Big)\text{d}W_s\Bigg)
			\end{align*} 
			for $t\in [0,1]$. Now, by the linearity of $f$ in the second argument
			\begin{align*}
				\Delta_n^{\frac 1 2}& \sum_{i=1}^{\left\lfloor t/\Delta_n \right\rfloor} \mathbb E\Big[f(\beta_i^n;\ V_i^n(1)) \left|\ \mathcal F_{(i-1)\Delta_n}\right.\Big] \\ &=
				\Delta_n \sum_{i=1}^{\left\lfloor t/\Delta_n \right\rfloor}\left(\mu_{(i-1)\Delta_n} \left.\mathbb E\Big[f(\{x\, W_s;\ s\in [0,1]\};\ \{s;\ s\in [0, 1]\})\Big]\right|_{x=\sigma_{(i-1)\Delta_n}}\right.  \\ &\quad \quad\quad \left.+ \
				\frac 1 2\, {\tilde \sigma}_{(i-1)\Delta_n} \left.\mathbb E\Big[f(\{x\, W_s;\ s\in [0,1]\};\ \{W_s^2-s;\ s\in [0, 1]\})\Big]\right|_{x=\sigma_{(i-1)\Delta_n}} \right) \\ &=
				\Delta_n \sum_{i=1}^{\left\lfloor t/\Delta_n \right\rfloor}\left(\mu_{(i-1)\Delta_n}\, \rho_{\sigma_{(i-1)\Delta_n}}^{(2)}(f) + \frac 1 2\,  {\tilde \sigma}_{(i-1)\Delta_n} \left(\rho_{\sigma_{(i-1)\Delta_n}}^{(3)}(f) - \rho_{\sigma_{(i-1)\Delta_n}}^{(2)}(f)\right)\right) \\ &\ucp
				\int_0^t \mu_s\, \rho_{\sigma_s}^{(2)}(f) \,\text{d}s + \frac 1 2 \int_0^t \tilde \sigma_s\, \rho_{\sigma_s}^{(3)}(f) \,\text{d}s  - \frac 1 2 \int_0^t \tilde \sigma_s\, \rho_{\sigma_s}^{(2)}(f)\,\text{d}s, 
			\end{align*}
			where we used the independence of $W$ and $V$. Due to linearity we observe the identity
			\[
			f(\beta_i^n;\ V_i^n(2)) = f(\beta_i^n;\ V_i^n(2)/\lVert V_i^n(2)\rVert _\infty)\lVert V_i^n(2)\rVert _\infty,
\]
whenever $\lVert V_i^n(2)\rVert _\infty>0$ and $0$ otherwise. Hence, we deduce that
			\begin{align*}
				\Delta_n^{\frac 1 2}& \sum_{i=1}^{\left\lfloor t/\Delta_n \right\rfloor} \mathbb E\Big[\big|f(\beta_i^n;\ V_i^n(2))\big| \Big] \\ &= 
				\Delta_n^{\frac 1 2} \sum_{i=1}^{\left\lfloor t/\Delta_n \right\rfloor} \mathbb E\Big[\big|f(\beta_i^n;\ V_i^n(2)/\lVert V_i^n(2)\rVert _\infty)\big|\lVert V_i^n(2)\rVert _\infty \Big] \\ &\le 
				\Delta_n^{\frac 1 2} C \sum_{i=1}^{\left\lfloor t/\Delta_n \right\rfloor}  \mathbb E\Big[\big\lVert V_i^n(2)\big\rVert _\infty^2 \Big]^{\frac 1 2} \\ & \le 
				C \bigg(\sum_{i=1}^{\left\lfloor t/\Delta_n \right\rfloor} \mathbb E\Big[\big\lVert V_i^n(2)\big\rVert _\infty^2 \Big]\bigg)^{\frac 1 2} \rightarrow 
				0
			\end{align*}
			by the Cauchy-Schwarz inequality and the polynomial growth of $f$. So we are only left to prove that $R^{2.2, n} \stackrel{ucp}{\rightarrow}  0$. Defining $\xi_i^n:=\frac{\Delta_n^{-\frac 1 2} d_i^n(X) -\beta_i^n}{\lVert \Delta_n^{-\frac 1 2} d_i^n(X) -\beta_i^n\rVert _\infty}$ for $\lVert \Delta_n^{-\frac 1 2} d_i^n(X) -\beta_i^n\rVert _\infty>0$ and $0$ otherwise, we get
			\begin{align*}
				\big|R^{2.2, n}_t\big| \le \Delta_n^{\frac 1 2} \sum_{i=1}^{\left\lfloor t/\Delta_n \right\rfloor} \mathbb E\left[\big|f(\chi_i^n;\ \xi_i^n)-f(\beta_i^n;\ \xi_i^n)\big|\big\lVert \Delta_n^{-\frac 1 2}\, d_i^n(X) -\beta_i^n\big\rVert _\infty \left|\ \mathcal F_{(i-1) \Delta_n}\right.\right].
			\end{align*}
			Therefore,
			\begin{align*}
				\mathbb P\Big[\sup_{t\le T} &\big|R^{2.2, n}_t\big| > \epsilon\Big]  \\ &\le 
				\frac{\Delta_n^{\frac 1 2}}{\epsilon} \sum_{i=1}^{\left\lfloor T/\Delta_n \right\rfloor} \mathbb E\Big[\big|f(\chi_i^n;\ \xi_i^n)-f(\beta_i^n;\ \xi_i^n)\big|\big\lVert \Delta_n^{-\frac 1 2}\, d_i^n(X) -\beta_i^n\big\rVert _\infty\Big] \\ &\le
				C\frac{\Delta_n }{\epsilon} \sum_{i=1}^{\left\lfloor T/\Delta_n \right\rfloor} \sqrt{\mathbb E\Big[\big(f(\chi_i^n;\ \xi_i^n)-f(\beta_i^n;\ \xi_i^n)\big)^2\Big]} \rightarrow
				0,
			\end{align*}
			where again we used Jensen's and Cauchy-Schwarz inequality and the local H\"older continuity of $f$ (uniformly on the unit circle in the second argument). Putting everything together, we have thus proven \textit{(iii)}. \qed
		
		\item[\textit{(iv)}] We want to show that 
			\begin{align*}
				R^{3, n}_t = \Delta_n^{-\frac 1 2} \sum_{i=1}^{\left\lfloor t/\Delta_n\right\rfloor} \left(\Delta_n \mathbb E\left[g(\beta_i^n)\left|\ \mathcal F_{(i-1) \Delta_n}\right.\right] - \int_{(i-1)\Delta_n}^{i\Delta_n} \rho_{\sigma_s}(g)\,\text{d}s\right) 
\ucp 0.
			\end{align*}
			Define $\mu_i^n:=\Delta_n^{-\frac 1 2}\int_{(i-1)\Delta_n}^{i\Delta_n}(\rho_{\sigma_{(i-1)\Delta_n}}(g)-\rho_{\sigma_s}(g))\, \text{d}s$ so $R^{3, n}_t=\sum_{i=1}^{\left\lfloor t/\Delta_n\right\rfloor} \mu_i^n$. Thanks to the differentiability of $g$ and 
the polynomial growth of $g$ and $g'$, we find that 
			\begin{align*}
				\lim_{h\rightarrow 0} \frac{\rho_{x+h}(g)-\rho_x(g)}h = 
				\E[g'_{W}(xW)],		
			\end{align*}
			so the derivative of $\rho_x(g)=:\psi(x)$ exists. Similarly, using that $g'$ is continuous, linear in the second argument and of polynomial growth in the first argument, we see that $\psi'$ is continuous. This allows us to expand $\mu_i^n=:\mu_i^n(1) + \mu_i^n(2)$ where
			\begin{align*}
				\mu_i^n(1)&= \Delta_n^{-\frac 1 2}\, \psi'(\sigma_{(i-1)\Delta_n})\int_{(i-1)\Delta_n}^{i\Delta_n}(\sigma_{(i-1)\Delta_n}-\sigma_s) \,\text{d}s, \\ 
				\mu_i^n(2)&= \Delta_n^{-\frac 1 2} \int_{(i-1)\Delta_n}^{i\Delta_n}(\psi'(\chi_s^n)-\psi'(\sigma_{(i-1)\Delta_n}))(\sigma_{(i-1)\Delta_n}-\sigma_s) \,\text{d}s
			\end{align*}
			with $\left|\chi_s^n-\sigma_{(i-1)\Delta_n}\right| \le \left|\sigma_{i\Delta_n}-\sigma_{(i-1)\Delta_n}\right|$. Now, decompose $-\mu_i^n(1)$ into a martingale increment $\mu_i^n(1.2)$ and a remainder term $\mu_i^n(1.1)$, i.e. 
$-\mu_i^n(1)=\mu_i^n(1.1) + \mu_i^n(1.2)$ with 
			\begin{align*}
				\mu_i^n(1.1)&= \Delta_n^{-\frac 1 2} \psi'(\sigma_{(i-1)\Delta_n})\int_{(i-1)\Delta_n}^{i\Delta_n}\left(\int_{(i-1)\Delta_n}^s \tilde \mu_u\, \text{du}\right)\,\text{d}s, \\
				\mu_i^n(1.2)&= \Delta_n^{-\frac 1 2} \psi'(\sigma_{(i-1)\Delta_n})\int_{(i-1)\Delta_n}^{i\Delta_n}\left(\int_{(i-1)\Delta_n}^s \tilde \sigma_u \,\text{d}W_u+\int_{(i-1)\Delta_n}^s \tilde v_u\, \text{d}V_u\right)\,\text{d}s.
			\end{align*}
			Observing that $\mu_i^n(1.1)\le \Delta_n^{\frac 3 2}\, \sup_{\left|x\right|\le \left|\sigma\right|}\psi'(x)\,  \lVert \tilde \mu\rVert_\infty$, its convergence to $0$ follows immediately. With the help of Doob's inequality,
			\begin{align*}
				\mathbb P\Big[\sup_{t\le T}\Big|\sum_{i=1}^{\left\lfloor t/\Delta_n\right\rfloor}\mu_i^n(1.2)\Big|> \epsilon\Big] & \le 
				C/\epsilon^2\,  \mathbb E\Big[\Big(\sum_{i=1}^{\left\lfloor t/\Delta_n\right\rfloor}\mu_i^n(1.2)\Big)^2\Big] \\ &=
				C / \epsilon^2 \sum_{i=1}^{\left\lfloor t/\Delta_n\right\rfloor} \mathbb E\left[\left(\mu_i^n(1.2)\right)^2\right] \\ &\le
				C\frac{\Delta_n^{\frac 1 2}} {\epsilon^2} \sup_{\left|x\right|\le \|\sigma\|_{\infty}}\psi'(x)\big(\lVert\tilde \sigma\rVert^2 + \lVert\tilde v\rVert^2\big) \rightarrow 
				0
			\end{align*}
			so $\sum_{i=1}^{\left\lfloor t/\Delta_n\right\rfloor}\mu_i^n(1)\ \stackrel{ucp}{\rightarrow} 0$. Regarding $\mu_i^n(2)$, since $\psi'$ is uniformly continuous on $[-\lVert \sigma\rVert _\infty, \lVert \sigma\rVert _\infty]$ choose $\delta>0$ for a given $\epsilon>0$ such that for all $s,\ t\le T$ we have $\left|\sigma_s-\sigma_t\right|\le \delta\Rightarrow \left| \psi'(\sigma_s)- \psi'(\sigma_t)\right|\le \epsilon$. Now, 
			\begin{align*}
				\left|\mu_i^n(2)\right| &\le \Delta_n^{-\frac 1 2}\, \epsilon\int_{(i-1)\Delta_n}^{i\Delta_n} \left|\sigma_{(i-1)\Delta_n}-\sigma_s\right|\,\text{d}s \\ &
				+ 2\, \Delta_n^{-\frac 1 2}/\delta \sup_{\left|x\right|\le \|\sigma\|_{\infty}}\left|\psi'(x)\right| \int_{(i-1)\Delta_n}^{i\Delta_n} \left|\sigma_{(i-1)\Delta_n}-\sigma_s\right|^2\,\text{d}s,
				\intertext{leading to}
				\mathbb P\Big[\sup_{t\le T}\big|\sum_{i=1}^{\left\lfloor t/\Delta_n \right\rfloor} \mu_i^n(2)\big| > \hat \epsilon\Big] &\le			\Delta_n^{-\frac 1 2}\, \epsilon /\, \hat\epsilon \ \mathbb E\Big[\int_0^T \left|\sigma_{(i-1)\Delta_n}-\sigma_s\right|\,\text{d}s\Big] \\ &
				+ 2\,  \Delta_n^{-\frac 1 2} \sup_{\left|x\right|\le \|\sigma\|_{\infty}}\left|\psi'(x)\right| \mathbb E\Big[\int_0^T \left|\sigma_{(i-1)\Delta_n}-\sigma_s\right|^2\,\text{d}s \Big]/(\hat\epsilon \delta) \\ &\le
				 C_T\Big(\epsilon\,  /\, \hat\epsilon + \Delta_n^{\frac 1 2}\,  \sup_{\left|x\right|\le \|\sigma\|_{\infty}}\left|\psi'(x)\right|\,   /\, (\hat\epsilon  \delta)\Big),		
			\end{align*}
			where we used Fubini's theorem. So choosing first $\epsilon$ small and then $n$ large finishes the proof of \textit{(iv)}, the last step in the proof of \eqref{CLT}. 
	\end{itemize}
\end{proof}

\paragraph{Proof of Theorems \ref{HFTheor3} and \ref{HFTheor4}}

As in the previous proof we may assume without loss of generality that the processes $\mu ,\sigma , \tilde{\mu},\tilde{\sigma},
\tilde{v}$ as well as the jump sizes $J, \tilde{J}$ are uniformly bounded in $(\omega,t)$. This is again justified by a standard
localisation procedure (see e.g. \cite{bgjps06}). Moreover, by the same localisation procedure, we may assume without loss of
generality that the jump sizes $J$ are bounded from below, i.e.
\[
|J_i|>\epsilon, \qquad 1\leq i\leq N_t, 
\]
for some $\epsilon >0$. Now, let $I_i^n=[(i-1)\Delta_n, i \Delta_n]$ and $\Omega_n:=\{ \omega\in\Omega: \#\{j\in \mathbb N: T_j \in J_i^n\}\le 1\}$, where $T_j$ denotes the the arrival time of the $j$'th jump of the Poisson process $N$. We clearly have
\begin{align*}
	\lim_{n \rightarrow \infty} \mathbb P[\Omega_{n}]&= \mathbb P[\Omega] = 1.
\end{align*}
Note that  each interval $I_i^n$ contains at most one jump of $X$ on $\Omega_n$ and each jump is at least of size $\epsilon$.

\paragraph{Proof of Theorem \ref{HFTheor3}}

The assertion $R(X,p)_t^n \pn R(X,p)_t$ for $p=2$ has been already proved in \cite{kp12}, so we show the result for $p>2$.
We write $X_t=X^c + Z_t$, where $X^c$ denotes the continuous part $X$ and $Z$ stands for the jump part. We define 
\[
K_n=\{i \le \left\lfloor t/\Delta_n\right\rfloor : \exists T_j \in I_i^n\}.
\]
Note that the cardinality of $K_n$ is finite almost surely as it is bounded by $N_t$. We decompose the statistic $R(X,p)_t^n$
as 
\begin{align} \label{rndec}
R(X,p)_t^n = \sum_{i\in K_n^c}  \sup_{s,u\in I_i^n} |X_s^c - X_u^c|^p 
+ \sum_{i\in K_n}  \sup_{s,u\in I_i^n} |X_s - X_u|^p.
\end{align} 
By Burkholder-Davis-Gundy inequality we conclude that $\E[\sup_{s,u\in I_i^n} |X_s^c - X_u^c|^p]\leq C_p \Delta_n^{p/2}$, and since
$p>2$, we obtain the convergence
\[
\sum_{i\in K_n^c}  \sup_{s,u\in I_i^n} |X_s^c - X_u^c|^p \pn 0.
\] 
Moreover, on $\Omega_n$, we have that 
\[
\sum_{i\in K_n}  \sup_{s,u\in I_i^n} |Z_s - Z_u|^p   - R(X,p)_t = \sum_{i=N_{\Delta_n \lfloor t/\Delta_n \rfloor}}^{N_t} |J_i|^p \pn 0.
\] 
Finally, we obtain by mean value theorem that 
\begin{align*}
& \E \left | \sum_{i\in K_n} \sup_{s,u\in I_i^n} |Z_s - Z_u|^p - \sup_{s,u\in I_i^n} |X_s - X_u|^p   \right| \\[1.5 ex]
&\leq p \E \sum_{i\in K_n} \sup_{s,u\in I_i^n} \max(|Z_s - Z_u|,|X_s - X_u|)^{p-1} \sup_{s,u\in I_i^n} |X_s^c - X_u^c| \\[1.5 ex]
& \pn 0,
\end{align*}
since the set $K_n$ is finite. Due to $\Omega_n\rightarrow \Omega$, we thus conclude the assertion of Theorem \ref{HFTheor3}. \qed

\paragraph{Proof of Theorem \ref{HFTheor4}}  

(i) We again use the decomposition \eqref{rndec} of $R(X,p)_t^n$. It holds that 
\[
\Delta_n^{-1/2} \E \Big[\sum_{i\in K_n^c}  \sup_{s,u\in I_i^n} |X_s^c - X_u^c|^p \Big]\leq \Delta_n^{(p-1)/2 -1}\rightarrow 0,
\]
since $p>3$. On the other hand we also have
\[
\Delta_n^{-1/2} \left(\sum_{i\in K_n}  \sup_{s,u\in I_i^n} |Z_s - Z_u|^p   - R(X,p)_t \right ) \pn 0. 
\]
Now, since all jump sizes $|J_i|$ are bounded by $\epsilon$ from below and $X^c$ is continuous, we obtain by mean value theorem
for all $i\in K_n$ on $\Omega_n$:
\begin{align*}
&\sup_{s,u\in I_i^n} (X_s - X_u)^p - \sup_{s,u\in I_i^n} (Z_s - Z_u)^p \\[1.5 ex]
&= p |\Delta Z_{T_i}|^{p-1} \left(
\sup_{\substack{s,u \in I_i^n\\ s<T_i^n\leq u}} (X^c_u-X^c_s) \boldsymbol 1_{\{\Delta Z_{T_i}>0\}}
+ \sup_{\substack{s,u \in I_i^n\\ u<T_i^n\leq s}} (X^c_u-X^c_s) \boldsymbol 1_{\{\Delta Z_{T_i}<0\}}  \right) \\[1.5 ex]
&+ o_{\mathbb P} (\Delta_n^{1/2}),
\end{align*} 
where $T_i^n$ denotes the jump time of $N_t$ in the interval $I_i^n$. Hence, we deduce the decomposition
\[
\Delta_n^{-1/2} \sum_{i\in K_n} \Big( \sup_{s,u\in I_i^n} (X_s - X_u)^p - \sup_{s,u\in I_i^n} (Z_s - Z_u)^p \Big)
= \sum_{i\in K_n} (\zeta_i^n + \overline{\zeta}_i^n) + o_{\mathbb P} (1),
\]
where 
\begin{align*}
\zeta_i^n &= p \Delta_n^{-1/2} |\Delta Z_{T_i^n}|^{p-1} \Big(
\sup_{\substack{s,u \in I_i^n\\ s<T_i^n\leq u}} (\sigma_{T_i^n -}(W_{T_i^n} - W_s) +
\sigma_{T_i^n }(W_u - W_{T_i^n})) \boldsymbol 1_{\{\Delta Z_{T_i^n}>0\}}  \\[1.5 ex]
&+ \sup_{\substack{s,u \in I_i^n\\ u<T_i^n\leq s}} (-\sigma_{T_i^n -}(W_{T_i^n} - W_u) -
\sigma_{T_i^n }(W_s - W_{T_i^n})) \boldsymbol 1_{\{\Delta Z_{T_i^n}<0\}}  \Big)
\end{align*}
and the quantity $\overline{\zeta}_i^n$ is defined via the identity
\[
\zeta_i^n + \overline{\zeta}_i^n = p \Delta_n^{-1/2}|\Delta Z_{T_i^n}|^{p-1} \Big(
\sup_{\substack{s,u \in I_i^n\\ s<T_i^n\leq u}} (X^c_u-X^c_s) \boldsymbol 1_{\{\Delta Z_{T_i^n}>0\}}
+ \sup_{\substack{s,u \in I_i^n\\ u<T_i^n\leq s}} (X^c_u-X^c_s) \boldsymbol 1_{\{\Delta Z_{T_i^n}<0\}}  \Big).
\]
Obviously the term $\zeta_i^n$ serves as the first order approximation while $\overline{\zeta}_i^n$ 
is the error term. Since $N$ and $W$ are independent, we obtain the stable convergence
\begin{align} 
(\kappa_i^n, \widetilde{W}_i^n)_{i\geq 1}&:=\Big( \Delta_n^{-1} \{T_i - \Delta_n \lfloor T_i/\Delta_n \rfloor \}, 
\Delta_n^{-1/2} \{W_{(i-1+s)\Delta_n} - W_{(i-1)\Delta_n}\}_{s\in [0,1]} \Big)_{i\geq 1} \nonumber \\
\label{helpweak} & \stab \Big(\kappa_i, \{\widetilde{W}_{i-1+s} - \widetilde{W}_{i-1}\}_{s\in [0,1]} \Big)_{i\geq 1},
\end{align} 
where $\kappa_i$ and $\widetilde{W}$ were defined in Section 4.2. This result is an immediate consequence 
of \cite[Lemma 6.2]{jp12}, but it can be easily shown in a straightforward manner.
Now, by properties of stable convergence and continuous mapping theorem,
we conclude that 
\[
\sum_{i\in K_n} \zeta_i^n \stab U(X,p)_t
\]
for any fixed $t>0$. Indeed this can be deduced from the stable convergence in \eqref{helpweak}, by defining the function $f_{i},
:
\R^3 \times [0,1] \times C([0,1])$ via
\begin{align*}
f_{i}(x,y,z) = p x_1 \sup_{0\leq s<y\leq u\leq 1} \Big( x_2(z(y)-z(s)) + x_3 (z(u)-z(y)) \Big) 
\end{align*} 
and observing that 
\begin{align*}
\zeta_i^n &= f_{i}\Big((|\Delta Z_{T_i^n}|^{p-1} \boldsymbol 1_{\{\Delta Z_{T_i^n}>0\}},
\sigma_{T_i^n -}, \sigma_{T_i^n }),\kappa_i^n,\widetilde{W}_i^n \Big) \\
&+ f_{i}\Big((|\Delta Z_{T_i^n}|^{p-1} \boldsymbol 1_{\{\Delta Z_{T_i^n}<0\}},
\sigma_{T_i^n -}, \sigma_{T_i^n }),\kappa_i^n,-\widetilde{W}_i^n \Big).
\end{align*}
Hence, to complete the proof we need to show that $\sum_{i\in K_n} \overline{\zeta}_i^n \pn 0$.
Since the processes $\mu ,\sigma$ and the jump sizes $J_i$ are uniformly bounded, we deduce
by Burkholder-Davis-Gundy inequality that
\[
\E[|\overline{\zeta}_i^n|^2] \leq C \Delta_n^{-1} \left( \Delta_n^2 + \int_{T_i^n}^{i\Delta_n} (\sigma_u-\sigma_{T_i^n})^2 du
+ \int_{(i-1)\Delta_n}^{T_i^n} (\sigma_u-\sigma_{T_i^n-})^2 du \right),  
\]
where the right side converges to 0, because $\sigma$ is c\'adl\'ag. This completes the proof since $K_n$ is finite. \qed 
\newline \newline
(ii)  Now let us consider the case $p=2$. According to the previous proof and the limiting results of \cite{cp07}
for the continuous case, we obtain the following asymptotic decomposition
\begin{align*}
\Delta_n^{-1/2} (R(X,2)_t^n - R(X,2)_t) = \sum_{i\in K_n} \zeta_i^n + \sum_{i\in K_n^c} \tilde{\zeta}_i^n + o_{\mathbb P}(1),
\end{align*}
where $\zeta_i^n$ has been defined in the previous step (now with $p=2$) and $\tilde{\zeta}_i^n$ serves as the first order 
approximation in the continuous case, i.e.
\[
\tilde{\zeta}_i^n = \Delta_n^{-1/2} \sigma_{(i-1)\Delta_n}^2 \Big( \sup_{s,u \in I_i^n} (W_u-W_s)^2 - \Delta_n^{-1}\lambda_{3,2} \Big).
\]
Now, we need to prove joint stable convergence of the vector $\Big(\sum_{i\in K_n} \zeta_i^n, 
 \sum_{i\in K_n^c} \tilde{\zeta}_i^n\Big)$. This problem is closely related to \cite[Lemma 5.8]{jacod08}. Indeed, following 
exactly the same proof steps, which are based on certain conditioning arguments,  it is sufficient to prove 
the stable central limit theorem for each component of the vector (indeed, the two stable limits are independent conditionally
on $\mathcal F$). But the stable convergence for the first component follows from the previous step
and the stable convergence for the second component has been shown in \cite{cp07} under the conditions of Theorem
\ref{HFTheor4}. This completes the proof. \qed

\section{References}

\bibliography{AfpdfocIs}

\begin{thebibliography}{10}

\bibitem{ds94}
F.~Delbaen and W.~Schachermayer, ``A general version of the fundamental theorem
  of asset pricing,'' {\em Mathematische Annalen}, vol.~300, pp.~463--520,
  1994.

\bibitem{bgjps06}
O.~E. Barndorff-Nielsen, S.~E. Graversen, J.~Jacod, M.~Podolskij, and
  N.~Shephard, ``A central limit theorem for realised power and bipower
  variations of continuous semimartingales.,'' {\em Kabanov Yu., R. Liptser,
  and J. Stoyanov (Eds.), From Stochastic Calculus to Mathematical Finance.
  Festschrift in Honour of A.N. Shiryaev, Heidelberg: Springer}, pp.~33--68,
  2006.

\bibitem{jacod08}
J.~Jacod, ``Asymptotic properties of realized power variations and related
  functionals of semimartingales,'' {\em Stochastic Processes and their
  Applications}, vol.~118, pp.~517--559, 2008.

\bibitem{kp08}
S.~Kinnebrock and M.~Podolskij, ``A note on the central limit theorem for
  bipower variation of general functions,'' {\em Stochastic Processes and Their
  Applications}, vol.~118, pp.~1056--1070, 2008.

\bibitem{jp12}
J.~Jacod and P.~Protter, {\em Discretization of Processes}.
\newblock Springer, 2012.

\bibitem{abm04}
T.~G. Andersen, T.~Bollerslev, and N.~Meddahi, ``Analytic evaluation of
  volatility forecasts,'' {\em International Economic Review}, vol.~45,
  pp.~1079--1110, 2004.

\bibitem{abm06}
T.~G. Andersen, T.~Bollerslev, and N.~Meddahi, ``Realized volatility
  forecasting and market microstructure noise,'' {\em Journal of Econometrics},
  2009.

\bibitem{aj09}
Y.~Ait-Sahalia and J.~Jacod, ``Testing for jumps in a discretely observed
  process,'' {\em The Annals of Statistics}, vol.~37, pp.~184--222, 2009.

\bibitem{bs06}
O.~E. Barndorff-Nielsen and N.~Shephard, ``Econometrics of testing for jumps in
  financial economics using bipower variation.,'' {\em Financial Econometrics},
  vol.~4, pp.~1--30, 2007.

\bibitem{n88}
J.~Norris, ``Intensity fluctuations of randomly scattered waves considered as
  local time,'' {\em Physics Letters A}, vol.~128, pp.~404--405, 1988.

\bibitem{f12}
J.~Fletcher, ``Sensing and multiscale structure,'' {\em
  http://arxiv.org/pdf/1204.0666.pdf}, 2012.

\bibitem{r63}
D.~Ray, ``Sojourn times of a diffusion process,'' {\em Illinois Journal of
  Mathematics}, vol.~7, pp.~615--630, 1963.

\bibitem{k63}
F.~Knight, ``Random walks and a sojourn density process of brownian motion,''
  {\em Transactions of the American Mathematical Society}, vol.~109,
  pp.~56--86, 1963.

\bibitem{f13}
J.~Fletcher, ``A limit theorem for the sum of squared differences of an
  integrated ito process with applications to inverse scattering,'' {\em
  http://arxiv.org/pdf/1211.6413.pdf}, 2013.

\bibitem{gk80}
M.~Garman and M.~J. Klass, ``On the estimation of security price volatilities
  from historical data,'' {\em Journal of Business}, vol.~53, pp.~67--78, 1980.

\bibitem{parkinson80}
M.~Parkinson, ``The extreme value method for estimating the variance of the
  rate of return,'' {\em Journal of Business}, vol.~53, pp.~61--65, 1980.

\bibitem{feller51}
W.~Feller, ``The asymptotic distribution of the range of sums of independent
  random variables,'' {\em The Annals of Mathematical Statistics}, vol.~22,
  pp.~427--432, 1951.

\bibitem{cp07}
K.~Christensen and M.~Podolskij, ``Realised range-based estimation of
  integrated variance,'' {\em Journal of Econometrics}, vol.~141, pp.~323--349,
  2007.

\bibitem{kp12}
K.~Christensen and M.~Podolskij, ``Asymptotic theory of range-based multipower
  variation,'' {\em Journal of Financial Econometrics}, vol.~10(3),
  pp.~417--456, 2012.

\bibitem{renyi63}
A.~Renyi, ``On stable sequences of events,'' {\em Sankhya}, vol.~25,
  pp.~293--302, 1963.

\bibitem{jacod94}
J.~Jacod, ``Limit of random measures associated with the increments of a
  brownian semimartingale,'' {\em Preprint number 120, Laboratoire de
  Probabiliti´es, Univ. P. et M. Curie}, 1994.

\bibitem{JS}
J.~Jacod and A.~Shirayev, {\em Limit Theorems for Stochastic Processes (2d
  ed.)}.
\newblock Springer, Berlin, 2003.

\end{thebibliography}
\bibliographystyle{ieeetr}

\end{document}